\documentclass{article}
\usepackage[utf8]{inputenc}

\usepackage{geometry}

\usepackage{float}
\usepackage{caption}
\usepackage{graphicx} 

\usepackage{amsthm}
\usepackage{amsmath}
\usepackage{amssymb}
\usepackage{mathrsfs}
\usepackage{amsfonts}
\usepackage{hyperref}
\usepackage{lineno}
\usepackage{fancyhdr}
\usepackage{dsfont}
\usepackage{indentfirst}
\usepackage{url}
\usepackage{cleveref}
\usepackage{enumerate}
\usepackage{listings}
\usepackage[dvipsnames]{xcolor}
\usepackage{nomencl}
\makenomenclature

\newtheorem{theorem}{Theorem}[section]

\newtheorem{lemma}[theorem]{Lemma}

\newtheorem{conjecture}[theorem]{Conjecture}

\theoremstyle{definition}

\usepackage{subfigure}
\usepackage{tikz}
\usetikzlibrary{decorations.markings}
\usetikzlibrary{backgrounds}

\usetikzlibrary{shapes, arrows, calc, arrows.meta, fit, positioning} 
\tikzset{  
	-stealth,auto,node distance =0.8 cm and 1 cm, thin, 
	state/.style ={circle, draw, inner sep=0.2pt}, 
	point/.style = {circle, draw, inner sep=0.18cm, fill, node contents={}},  
	el/.style = {inner sep=2pt, align=right, sloped}  
}

\title{Arc-disjoint in- and out-branchings in semicomplete split digraphs}

\date{\today}

\author{Jiangdong Ai\thanks{Corresponding author. School of Mathematical Sciences and LPMC, Nankai University, Tianjin 300071, P.R.
China. Email: jd@nankai.edu.cn. Partially supported
by the Fundamental Research Funds for the Central Universities, Nankai University. },~ Yiming Hao\thanks{School of Mathematical Sciences and LPMC, Nankai University, Tianjin 300071, P.R.
China. Email: 1120230031@mail.nankai.edu.cn.},~ Zhaoxiang Li\thanks{School of Mathematical Sciences and LPMC, Nankai University, Tianjin 300071, P.R.
China. Email: ZhaoxiangLi@mail.nankai.edu.cn.},~ Qi Shao\thanks{Center for Combinatorics, Nankai University, Tianjin 300071, P.R China. Email: Shaoqi@mail.nankai.edu.cn.}}

\begin{document}
	
	\maketitle
 \begin{abstract}
 An \emph{out-tree (in-tree)} is an oriented tree where every vertex except one, called the \emph{root}, has in-degree (out-degree) one. An \emph{out-branching $B^+_u$ (in-branching $B^-_u$)} of a digraph $D$ is a spanning out-tree (in-tree) rooted at $u$. A \emph{good $(u,v)$-pair} in $D$ is a pair of branchings $B^+_u, B^-_v$ which are arc-disjoint. Thomassen proved that deciding whether a digraph has any good pair is NP-complete. A \emph{semicomplete split digraph} is a digraph where the vertex set is the disjoint union of two non-empty sets, $V_1$ and $V_2$, such that $V_1$ is an independent set, the subdigraph induced by $V_2$ is semicomplete, and every vertex in $V_1$ is adjacent to every vertex in $V_2$. In this paper, we prove that every $2$-arc-strong semicomplete split digraph $D$ contains a good $(u, v)$-pair for any choice of vertices $u, v$ of $D$, thereby confirming a conjecture by Bang-Jensen and Wang [Bang-Jensen and Wang, J. Graph Theory, 2024].
 \end{abstract}

\section{Introduction}

The notation follows \cite{ai2024}, so we only repeat a few definitions here. A digraph is not allowed to have parallel arcs or loops. A directed multigraph can have parallel arcs but no loops. A directed multigraph is called \emph{semicomplete} if there are no pair of non-adjacent vertices. A directed multigraph $D$ is \emph{$k$-arc-strong} if $D\setminus{}A^{\prime}$ remains strong for every subset $A^{\prime} \subseteq A(D)$ of size at most $k - 1$.

An \emph{out‐branching} (resp., \emph{in‐branching}) of $D$ is a spanning oriented tree in which every vertex except one, called the root, has in-degree (resp., out-degree) one in $D$. A \emph{good $(u,v)$-pair} in $D$ is a pair of arc-disjoint out-branching and in-branching rooted at $u$ and $v$ respectively. We say a vertex pair $(u,v)$ is a \emph{good vertex pair} if there is a good $(u,v)$-pair in $D$. Thomassen \cite{thomassen1989} proposed the following conjecture.

\begin{conjecture}\cite{thomassen1989}\label{conj:CTbranch}
    There exists an integer $K$ such that every $K$-arc-strong digraph $D=(V,A)$   has a good $(u,v)$-pair for every choice of $u,v\in V$.
\end{conjecture}

A \emph{strong arc decomposition} of a (multi-)digraph $D=(V,A)$ is a partition of its arc set $A$ into two disjoint arc sets $A_1$ and $A_2$ such that both of the spanning subdigraphs $D_1=(V, A_1)$ and $D_2=(V, A_2)$ are strong. Clearly, if a digraph $D$ has a strong arc decomposition, then $D$ contains a good $(u,v)$-pair for every choice of $u,v\in V$. Therefore, the following conjecture by Bang-Jensen and Yeo would imply Conjecture~\ref{conj:CTbranch}.
	
\begin{conjecture}\cite{bangCOM24}
    There exists an integer $K$ such that every $K$-arc-strong digraph has a strong arc decomposition.
\end{conjecture}

We can easily see that every digraph $D$ with a strong arc decomposition is 2-arc-strong. Then, asking if every 2-arc-strong digraph has a strong arc decomposition is natural. Unfortunately, the following digraphs provide a negative answer.

Bang-Jensen and Yeo~\cite{bangCOM24} proved that for a 2-arc-strong semicomplete digraph $D$, there is only one exception that does not have a strong arc decomposition, which is the digraph $S_4$ depicted in Figure \ref{fig-MD-exceptional}. Bang-Jensen, Gutin and Yeo~\cite{bangJGT95} generalized the above result to semicomplete multi-digraphs with six additional exceptions, see Figure~\ref{fig-MD-exceptional}.

	\begin{theorem}\footnote{This is a modified version, the authors of \cite{bangJGT95} missed $S_{4,4}, S_{4,5}$ and $S_{4,6}$.}\cite{bangJGT95}\label{thm:multi semi SAD}
		A 2-arc-strong semicomplete multi-digraph $D=(V,A)$ on $n$ vertices has a strong arc decomposition if and only if $D$ is not isomorphic to one of the exceptional digraphs depicted in Figure \ref{fig-MD-exceptional}. Furthermore, a strong arc decomposition of $D$ can be obtained in polynomial time when it exists.	
	\end{theorem}
	\begin{figure}[H]
		\centering
		\subfigure{\begin{minipage}[t]{0.23\linewidth}
				\centering\begin{tikzpicture}[scale=0.8]
					\filldraw[black](0,0) circle (3pt)node[label=left:$v_1$](v1){};
					\filldraw[black](2,0) circle (3pt)node[label=right:$v_2$](v2){};
					\filldraw[black](0,-2) circle (3pt)node[label=left:$v_3$](v3){};
					\filldraw[black](2,-2) circle (3pt)node[label=right:$v_4$](v4){};
					\foreach \i/\j/\t in {
						v1/v2/0,
						v2/v3/0,
						v3/v4/0,
						v4/v1/0,
						v1/v3/15,
						v2/v4/15,
						v3/v1/15,
						v4/v2/15
					}{\path[draw, line width=0.8] (\i) edge[bend left=\t] (\j);}			
				\end{tikzpicture}\caption*{$S_4$}\end{minipage}}
		\subfigure{\begin{minipage}[t]{0.23\linewidth}
				\centering\begin{tikzpicture}[scale=0.8]
					\filldraw[black](0,0) circle (3pt)node[label=left:$v_1$](v1){};
					\filldraw[black](2,0) circle (3pt)node[label=right:$v_2$](v2){};
					\filldraw[black](0,-2) circle (3pt)node[label=left:$v_3$](v3){};
					\filldraw[black](2,-2) circle (3pt)node[label=right:$v_4$](v4){};
					\foreach \i/\j/\t in {
						v1/v2/0,
						v2/v3/0,
						v3/v4/0,
						v4/v1/0,
						v1/v3/15,
						v2/v4/15,
						v3/v1/10,
						v1/v3/-30,
						v4/v2/15
					}{\path[draw, line width=0.8] (\i) edge[bend left=\t] (\j);}	
				\end{tikzpicture}\caption*{$S_{4,1}$}\end{minipage}}
		\subfigure{\begin{minipage}[t]{0.23\linewidth}
				\centering\begin{tikzpicture}[scale=0.8]
					\filldraw[black](0,0) circle (3pt)node[label=left:$v_1$](v1){};
					\filldraw[black](2,0) circle (3pt)node[label=right:$v_2$](v2){};
					\filldraw[black](0,-2) circle (3pt)node[label=left:$v_3$](v3){};
					\filldraw[black](2,-2) circle (3pt)node[label=right:$v_4$](v4){};
					\foreach \i/\j/\t in {
						v1/v2/0,
						v1/v2/20,
						v2/v3/0,
						v3/v4/0,
						v4/v1/0,
						v1/v3/15,
						v2/v4/15,
						v3/v1/15,
						v4/v2/15
					}{\path[draw, line width=0.8] (\i) edge[bend left=\t] (\j);}	
				\end{tikzpicture}\caption*{$S_{4,2}$}\end{minipage}}
		\subfigure{\begin{minipage}[t]{0.23\linewidth}
				\centering\begin{tikzpicture}[scale=0.8]
					\filldraw[black](0,0) circle (3pt)node[label=left:$v_1$](v1){};
					\filldraw[black](2,0) circle (3pt)node[label=right:$v_2$](v2){};
					\filldraw[black](0,-2) circle (3pt)node[label=left:$v_3$](v3){};
					\filldraw[black](2,-2) circle (3pt)node[label=right:$v_4$](v4){};
					\foreach \i/\j/\t in {
						v1/v2/0,
						v2/v3/0,
						v3/v4/0,
						v4/v1/0,
						v1/v3/15,
						v2/v4/10,
						v4/v2/-30,
						v3/v1/10,
						v1/v3/-30,
						v4/v2/15
					}{\path[draw, line width=0.8] (\i) edge[bend left=\t] (\j);}	
				\end{tikzpicture}\caption*{$S_{4,3}$}\end{minipage}}
		
		\subfigure{\begin{minipage}[t]{0.23\linewidth}
				\centering\begin{tikzpicture}[scale=0.8]
					\filldraw[black](0,0) circle (3pt)node[label=left:$v_1$](v1){};
					\filldraw[black](2,0) circle (3pt)node[label=right:$v_2$](v2){};
					\filldraw[black](0,-2) circle (3pt)node[label=left:$v_3$](v3){};
					\filldraw[black](2,-2) circle (3pt)node[label=right:$v_4$](v4){};
					\foreach \i/\j/\t in {
						v1/v2/0,
						v1/v2/20,
						v2/v3/0,
						v3/v4/0,
						v4/v1/0,
						v1/v3/15,
						v2/v4/15,
						v3/v1/10,
						v1/v3/-30,
						v4/v2/15
					}{\path[draw, line width=0.8] (\i) edge[bend left=\t] (\j);}	
				\end{tikzpicture}\caption*{$S_{4,4}$}\end{minipage}}
		\subfigure{\begin{minipage}[t]{0.23\linewidth}
				\centering\begin{tikzpicture}[scale=0.8]
					\filldraw[black](0,0) circle (3pt)node[label=left:$v_1$](v1){};
					\filldraw[black](2,0) circle (3pt)node[label=right:$v_2$](v2){};
					\filldraw[black](0,-2) circle (3pt)node[label=left:$v_3$](v3){};
					\filldraw[black](2,-2) circle (3pt)node[label=right:$v_4$](v4){};
					\foreach \i/\j/\t in {
						v1/v2/0,
						v1/v2/20,
						v2/v3/0,
						v3/v4/0,
						v4/v1/0,
						v1/v3/15,
						v2/v4/10,
						v3/v1/15,
						v4/v2/-30,
						v4/v2/15
					}{\path[draw, line width=0.8] (\i) edge[bend left=\t] (\j);}	
				\end{tikzpicture}\caption*{$S_{4,5}$}\end{minipage}}
		\subfigure{\begin{minipage}[t]{0.23\linewidth}
				\centering\begin{tikzpicture}[scale=0.8]
					\filldraw[black](0,0) circle (3pt)node[label=left:$v_1$](v1){};
					\filldraw[black](2,0) circle (3pt)node[label=right:$v_2$](v2){};
					\filldraw[black](0,-2) circle (3pt)node[label=left:$v_3$](v3){};
					\filldraw[black](2,-2) circle (3pt)node[label=right:$v_4$](v4){};
					\foreach \i/\j/\t in {
						v1/v2/0,
						v2/v3/0,
						v3/v4/0,
						v4/v1/0,
						v1/v3/15,
						v2/v4/10,
						v4/v2/-30,
						v3/v1/10,
						v1/v3/-30,
						v4/v2/15,
						v1/v2/20
					}{\path[draw, line width=0.8] (\i) edge[bend left=\t] (\j);}	
				\end{tikzpicture}\caption*{$S_{4,6}$}\end{minipage}}

		\caption{2-arc-strong multi-digraphs without strong arc decompositions.}
		\label{fig-MD-exceptional}
	\end{figure}
	Subsequently, Bang-Jensen and Huang~\cite{bangJCTB102} extended semicomplete digraphs to locally semicomplete digraphs. Bang-Jensen, Gutin and Yeo \cite{bangJGT95} considered the strong 
	arc decomposition of semicomplete composition and solved it completely.

    A \emph{split digraph} $D$ is a digraph whose vertex set is a disjoint union of two non-empty sets $V_1$ and $V_2$ such that $V_1$ is an independent set, the subdigraph induced by $V_2$ is semicomplete, and we denote it as $D=\left(V_1, V_2; A\right)$. Additionally, if every vertex in $V_1$ is adjacent to every vertex in $V_2$, we call such a split digraph \emph{semicomplete split digraph}. Recently, Bang-Jensen and Wang \cite{bang2024} explored the strong arc decomposition of split digraphs. Their main result is the following:
	\begin{theorem}\cite{bang2024}
		Let $D=\left(V_1, V_2; A\right)$ be a 2-arc-strong split digraph such that $V_1$ is an independent set and the subdigraph induced by $V_2$ is semicomplete. If every vertex of $V_1$ has both out- and in-degree at least 3 in $D$, then $D$ has a strong arc decomposition.
	\end{theorem}

Later, Ai, He, Li, Qin and Wang~\cite{ai2024} provided a complete characterization of whether a split digraph has a strong arc decomposition.

\begin{theorem}\cite{ai2024}\label{thm:2as}
    A 2-arc-strong split digraph $D=(V_1,V_2;A)$ has a strong arc decomposition if and only if $D$  is not isomorphic to any of the digraphs illustrated as follows, or their arc-reversed versions (reverse all arcs). 
    \begin{enumerate}
    \renewcommand{\theenumi}{\textbf{(\arabic{enumi})}}
        \item\label{ce1} There are $x_1,x_2,u\in V(D)$ such that $N_D^+(u)=\{x_1,x_3\},$ $N_D^-(u)=\{x_1,x_2\},$ $N_D^+(x_1)=\{x_2,u\}, N_D^+(x_2)=\{v,u\},$ where $x_3,v\in V(D)\backslash\{x_1,x_2,u\}$ and $x_3,v$ can be the same.

        \item\label{ce2} There exist $a,b,c \in V_2, u,v\in V_1$, such that $N_D^+(b)=\{u,v,c\}, N_D^+(c)=\{v,a\}, N_D^+(a)=\{u,b\}, N_D^+(u)=\{a,u^+\}, N_D^+(v)=\{b,v^+\}$, where $u^+,v^+\in V_2\setminus \{a,b,c\}$ and $u^+,v^+$ can be the same, besides, $N_D^-(u)=\{a,b\}, N_D^-(v)=\{b,c\}$.

        \item\label{ce3} The Appendix in~\cite{ai2024}.
    \end{enumerate}   
\end{theorem}   

			

			
\begin{figure}[H]
    \centering
    \hfill
    \subfigure{\begin{minipage}[t]{0.45\linewidth}
				\centering\begin{tikzpicture}[scale=0.3]
				
				\filldraw[black](15,0) circle (5pt)node[label=right:$x_1$](x1){};
				\filldraw[black](10,0) circle (5pt)node[label=above:{$x_2$}](x2){};
				\filldraw[black](6,-1.5) circle (5pt)node[label=above:{$v$}](v){};
				\filldraw[black](12.5,4.5) circle (5pt)node[label=right:{$u$}](u){};
				\filldraw[black](3,1.5) circle (5pt)node[label=above:{$x_3$}](x3){};
				
				\path[draw, bend left=30, line width=0.8pt] (x1) to (x2);
				\path[draw, bend left=30, line width=0.8pt] (x2) to (v);
				\path[draw, bend left=30, line width=0.8pt] (x1) to (u);
				\path[draw, line width=0.8pt] (x2) to (u);
				\path[draw, bend left=30, line width=0.8pt] (u) to (x1);	
				\path[draw, line width=0.8pt] (u) to (x3);	
				
			\end{tikzpicture}\caption*{An illustration of counterexample in \ref{ce1}}\end{minipage}}
   \hfill
   \subfigure{\begin{minipage}[t]{0.45\linewidth}
				\centering\begin{tikzpicture}[scale=0.3]

			\filldraw[black](16,0) circle (5pt)node[label=right:$c$](c){};
			\filldraw[black](12.5,3) circle (5pt)node[label=right:{$a$}](a){};
			\filldraw[black](12.5,-3) circle (5pt)node[label=right:{$b$}](b){};
			\filldraw[black](7,2) circle (5pt)node[label=above:{$u$}](u){};
			\filldraw[black](7,-2) circle (5pt)node[label=below:{$v$}](v){};
			\filldraw[black](3,2) circle (5pt)node[label=above:{$u^+$}](u+){};
			\filldraw[black](3,-2) circle (5pt)node[label=below:{$v^+$}](v+){};
			
			\path[draw, line width=0.8pt, bend left=15] (a) to (u);
			\path[draw, line width=0.8pt] (u) to (u+);
			\path[draw, line width=0.8pt] (c) to (v);
			\path[draw, line width=0.8pt, bend left=15] (v) to (b);
			
			\path[draw, line width=0.8pt] (a) to (b);
			\path[draw, line width=0.8pt] (b) to (u);
			\path[draw, line width=0.8pt, bend left=15] (u) to (a);
			\path[draw, line width=0.8pt] (v) to (v+);
			\path[draw, line width=0.8pt, bend left=15] (b) to (v);
			\path[draw, line width=0.8pt] (c) to (a);
			
			\path[draw, line width=0.8pt] (b) to (c);
		\end{tikzpicture}\caption*{An illustration of counterexample in \ref{ce2}}\end{minipage}}
    \hfill
    
    \label{fig:2 kinds of counterexamples}
\end{figure}
 Note that if a digraph $D$ contains a strong arc decomposition, it has a good $(u,v)$-pair for every choice of vertices $u, v\in D$. From Theorem~\ref{thm:2as}, we see the condition $2$-arc-strong is not enough to guarantee a strong arc decomposition for a split digraph, then what if a good $(u,v)$-pair for every choice of $u,v$?  In~\cite{bang2024}, Bang-Jensen and Wang proved that there are infinitely many $2$-arc-strong split digraphs that do not have good $(u,v)$-pairs for some choice of $u,v$. Now, it is natural to ask for semicomplete split digraphs, and they proposed the following conjecture.
\begin{conjecture}\cite{bang2024}\label{conj:main}
    Every 2-arc-strong semicomplete split digraph $D$ contains a good $(u,v)$-pair for every choice of vertices $u,v$ in $D$.
\end{conjecture}
In this paper, we proved that every $2$-arc-strong semicomplete split digraph $D$ contains a good $(u, v)$-pair for every choice of vertices $u, v$ of $D$, which confirms the Conjecture~\ref{conj:main}. 

\begin{theorem}\label{thm:main}
    Every 2-arc-strong semicomplete split digraph $D$ contains a good $(u,v)$-pair for every choice of vertices $u,v$ of $D$.
\end{theorem}

\section{The Proof of the Main Result}

Based on Theorem~\ref{thm:2as}, we need to consider the semicomplete split digraphs with the structure illustrated in \ref{ce1}, \ref{ce2}
and the semicomplete split digraphs in Appendix of \cite{ai2024}.

We first consider the semicomplete split digraphs in Appendix of \cite{ai2024}, and then discuss the semicomplete split digraphs with the structure illustrated in \ref{ce1} and \ref{ce2} of Theorem~\ref{thm:2as}. 

In the first part, the main idea is to consider the common subdigraphs of these cases. By finding special good pairs in $S_4$ and then transforming them into good pairs in $(\overline{i})-(\overline{v})$, we can verify that there is a good $(u,v)$-pair for every choice of $u,v$ in $(\overline{i})-(\overline{v})$. Then based on the results on $(\overline{i})-(\overline{v})$, we can verify that there is a good $(u,v)$-pair for every choice of $u,v$ in the combinations of the five basic cases by checking the pairs $(a,b)$ and $(b,a)$. In the second part, we verify that all the semicomplete split digraphs with structure in \ref{ce1} have been covered by those in the Appendix, and no semicomplete split digraphs with structure in \ref{ce2}. 

For the rest of this paper, we denote an out-branching (resp., in-branching) rooted at $t$ by $t$-out-branching (resp., $t$-in-branching) for simplicity.

\subsection{Five Basic Cases in Appendix}

The following figures show the structures of the five basic cases. We say they are 'basic' since any other cases on $6$ vertices can be somehow viewed as the combinations of these 5 cases.  

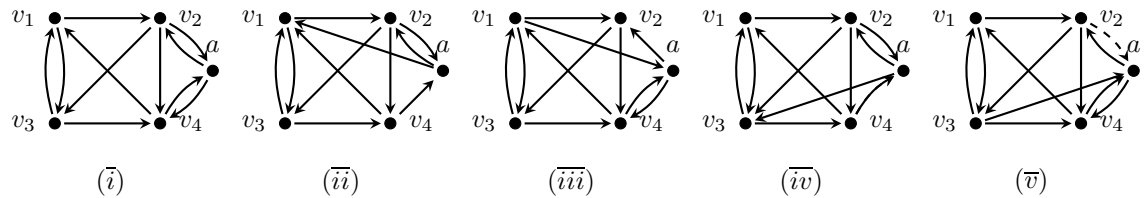
\begin{figure}[H]
		\begin{minipage}[t]{0.19\linewidth}
			\vspace{0pt}
			\centering
			\begin{tikzpicture}[scale=0.7]
				\filldraw[black](0,0) circle (3pt)node[label=left:$v_1$](v1){};
				\filldraw[black](2,0) circle (3pt)node[label=right:$v_2$](v2){};
				\filldraw[black](0,-2) circle (3pt)node[label=left:$v_3$](v3){};
				\filldraw[black](2,-2) circle (3pt)node[label=right:$v_4$](v4){};
				
				\filldraw[black](3,-1) circle (3pt)node[label=above:$a$](a){};
				
				\foreach \i/\j/\t in {
					v1/v2/0,
					v2/v3/0,
					v3/v4/0,
					v4/v1/0,
					v1/v3/15,
                        v3/v1/15,     
					v2/v4/0,
					a/v2/15,
					a/v4/15,
					v4/a/15,
					v2/a/15
				}{\path[draw, line width=0.8] (\i) edge[bend left=\t] (\j);}
			\end{tikzpicture}
			\caption*{$(\overline{i})$}
		\end{minipage}
		\hfill
		\begin{minipage}[t]{0.19\linewidth}
			\vspace{0pt}
			\centering
			\begin{tikzpicture}[scale=0.7]
				\filldraw[black](0,0) circle (3pt)node[label=left:$v_1$](v1){};
				\filldraw[black](2,0) circle (3pt)node[label=right:$v_2$](v2){};
				\filldraw[black](0,-2) circle (3pt)node[label=left:$v_3$](v3){};
				\filldraw[black](2,-2) circle (3pt)node[label=right:$v_4$](v4){};
				
				\filldraw[black](3,-1) circle (3pt)node[label=above:$a$](a){};
				
				\foreach \i/\j/\t in {
					v1/v2/0,
					v2/v3/0,
					v3/v4/0,
					v4/v1/0,
					v1/v3/15,
                        v3/v1/15,
					v2/v4/0,
					a/v2/15,
					v2/a/15,
					v4/a/0,
					a/v1/0
				}{\path[draw, line width=0.8] (\i) edge[bend left=\t] (\j);}
			\end{tikzpicture}
			\caption*{$(\overline{ii})$}
		\end{minipage}
		\hfill
		\begin{minipage}[t]{0.19\linewidth}
			\vspace{0pt}
			\centering
			\begin{tikzpicture}[scale=0.7]
				\filldraw[black](0,0) circle (3pt)node[label=left:$v_1$](v1){};
				\filldraw[black](2,0) circle (3pt)node[label=right:$v_2$](v2){};
				\filldraw[black](0,-2) circle (3pt)node[label=left:$v_3$](v3){};
				\filldraw[black](2,-2) circle (3pt)node[label=right:$v_4$](v4){};
				
				\filldraw[black](3,-1) circle (3pt)node[label=above:$a$](a){};
				
				\foreach \i/\j/\t in {
					v1/v2/0,
					v2/v3/0,
					v3/v4/0,
					v4/v1/0,
					v1/v3/15,
                        v3/v1/15,
					v2/v4/0,
					a/v2/0,
					a/v4/15,
					v4/a/15,
					v1/a/0
				}{\path[draw, line width=0.8] (\i) edge[bend left=\t] (\j);}
			\end{tikzpicture}
			\caption*{$(\overline{iii})$}
		\end{minipage}
		\hfill
		\begin{minipage}[t]{0.19\linewidth}
			\vspace{0pt}
			\centering
			\begin{tikzpicture}[scale=0.7]
				\filldraw[black](0,0) circle (3pt)node[label=left:$v_1$](v1){};
				\filldraw[black](2,0) circle (3pt)node[label=right:$v_2$](v2){};
				\filldraw[black](0,-2) circle (3pt)node[label=left:$v_3$](v3){};
				\filldraw[black](2,-2) circle (3pt)node[label=right:$v_4$](v4){};
				
				\filldraw[black](3,-1) circle (3pt)node[label=above:$a$](a){};
				
				\foreach \i/\j/\t in {
					v1/v2/0,
					v2/v3/0,
					v3/v4/0,
					v4/v1/0,
					v1/v3/15,
					v2/v4/0,
                        v3/v1/15,
					a/v2/15,
					v2/a/15,
					v4/a/15,
					a/v3/0
				}{\path[draw,  line width=0.8] (\i) edge[bend left=\t] (\j);}
    
                 \path[draw, dashed, line width=0.8] (v4) edge[bend left=15] (a);
			\end{tikzpicture}
			\caption*{$(\overline{iv})$}
		\end{minipage}
		\hfill
		\begin{minipage}[t]{0.19\linewidth}
			\vspace{0pt}
			\centering
			\begin{tikzpicture}[scale=0.7]
				\filldraw[black](0,0) circle (3pt)node[label=left:$v_1$](v1){};
				\filldraw[black](2,0) circle (3pt)node[label=right:$v_2$](v2){};
				\filldraw[black](0,-2) circle (3pt)node[label=left:$v_3$](v3){};
				\filldraw[black](2,-2) circle (3pt)node[label=right:$v_4$](v4){};
				
				\filldraw[black](3,-1) circle (3pt)node[label=above:$a$](a){};
				
				\foreach \i/\j/\t in {
					v1/v2/0,
					v2/v3/0,
					v3/v4/0,
					v4/v1/0,
					v1/v3/15,
					v2/v4/0,
                        v3/v1/15,
					a/v2/15,
					a/v4/15,
					v4/a/15,
					v3/a/0
				}{\path[draw, line width=0.8] (\i) edge[bend left=\t] (\j);}
				
				\path[draw, dashed, line width=0.8] (v2) edge[bend left=15] (a);
			\end{tikzpicture}
			\caption*{$(\overline{v})$}
		\end{minipage}
	\end{figure}

We need the following lemma. 

    
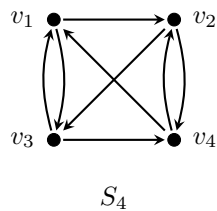
\begin{figure}[H]
		\centering
		\subfigure{\begin{minipage}[t]{0.23\linewidth}
				\centering\begin{tikzpicture}[scale=0.8]
					\filldraw[black](0,0) circle (3pt)node[label=left:$v_1$](v1){};
					\filldraw[black](2,0) circle (3pt)node[label=right:$v_2$](v2){};
					\filldraw[black](0,-2) circle (3pt)node[label=left:$v_3$](v3){};
					\filldraw[black](2,-2) circle (3pt)node[label=right:$v_4$](v4){};
					\foreach \i/\j/\t in {
						v1/v2/0,
						v2/v3/0,
						v3/v4/0,
						v4/v1/0,
						v1/v3/15,
						v2/v4/15,
						v3/v1/15,
						v4/v2/15
					}{\path[draw, line width=0.8] (\i) edge[bend left=\t] (\j);}			
				\end{tikzpicture}\caption*{$S_4$}\end{minipage}}
   \end{figure}

\begin{lemma}\label{lem21}
 Each vertex pair in $S_{4}$ is a good vertex pair.
\end{lemma}

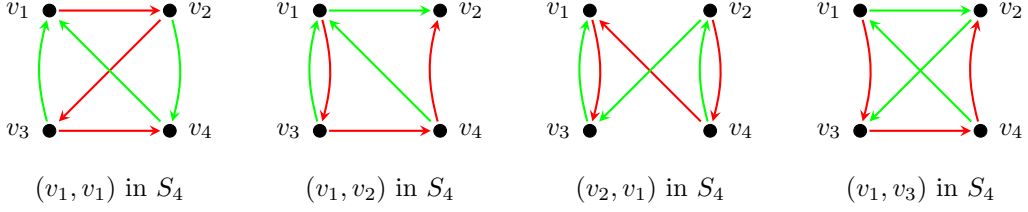
\begin{figure}[H]
		\centering
		\subfigure{\begin{minipage}[t]{0.23\linewidth}
				\centering\begin{tikzpicture}[scale=0.8]
					\filldraw[black](0,0) circle (3pt)node[label=left:$v_1$](v1){};
					\filldraw[black](2,0) circle (3pt)node[label=right:$v_2$](v2){};
					\filldraw[black](0,-2) circle (3pt)node[label=left:$v_3$](v3){};
					\filldraw[black](2,-2) circle (3pt)node[label=right:$v_4$](v4){};
					\foreach \i/\j/\t in {
						v1/v2/0,
						v2/v3/0,
						v3/v4/0
					}{\path[draw, line width=0.8, red] (\i) edge[bend left=\t] (\j);}
                        \foreach \i/\j/\t in {
						v2/v4/15,
                        v4/v1/0,
						v3/v1/15
					}{\path[draw, line width=0.8, green] (\i) edge[bend left=\t] (\j);}
				\end{tikzpicture}\caption*{$(v_1,v_1)$ in $S_4$}
    
    \end{minipage}}
		\subfigure{\begin{minipage}[t]{0.23\linewidth}
				\centering\begin{tikzpicture}[scale=0.8]
					\filldraw[black](0,0) circle (3pt)node[label=left:$v_1$](v1){};
					\filldraw[black](2,0) circle (3pt)node[label=right:$v_2$](v2){};
					\filldraw[black](0,-2) circle (3pt)node[label=left:$v_3$](v3){};
					\filldraw[black](2,-2) circle (3pt)node[label=right:$v_4$](v4){};
					\foreach \i/\j/\t in {
						v1/v3/15,
						v3/v4/0,
						v4/v2/15
					}{\path[draw, line width=0.8, red] (\i) edge[bend left=\t] (\j);};
                        \foreach \i/\j/\t in {
						v3/v1/15,
                        v4/v1/0,
						v1/v2/0
					}{\path[draw, line width=0.8, green] (\i) edge[bend left=\t] (\j);}
				\end{tikzpicture}\caption*{$(v_1,v_2)$ in $S_4$}\label{2}\end{minipage}}
        \subfigure{\begin{minipage}[t]{0.23\linewidth}
				\centering\begin{tikzpicture}[scale=0.8]
					\filldraw[black](0,0) circle (3pt)node[label=left:$v_1$](v1){};
					\filldraw[black](2,0) circle (3pt)node[label=right:$v_2$](v2){};
					\filldraw[black](0,-2) circle (3pt)node[label=left:$v_3$](v3){};
					\filldraw[black](2,-2) circle (3pt)node[label=right:$v_4$](v4){};
					\foreach \i/\j/\t in {
						v2/v4/15,
						v4/v1/0,
						v1/v3/15
					}{\path[draw, line width=0.8, red] (\i) edge[bend left=\t] (\j);};
                        \foreach \i/\j/\t in {
						v4/v2/15,
                        v2/v3/0,
						v3/v1/15
					}{\path[draw, line width=0.8, green] (\i) edge[bend left=\t] (\j);}
				\end{tikzpicture}\caption*{$(v_2,v_1)$ in $S_4$}\label{3}\end{minipage}}
        \subfigure{\begin{minipage}[t]{0.23\linewidth}
				\centering\begin{tikzpicture}[scale=0.8]
					\filldraw[black](0,0) circle (3pt)node[label=left:$v_1$](v1){};
					\filldraw[black](2,0) circle (3pt)node[label=right:$v_2$](v2){};
					\filldraw[black](0,-2) circle (3pt)node[label=left:$v_3$](v3){};
					\filldraw[black](2,-2) circle (3pt)node[label=right:$v_4$](v4){};
					\foreach \i/\j/\t in {
						v1/v3/15,
						v3/v4/0,
						v4/v2/15
					}{\path[draw, line width=0.8, red] (\i) edge[bend left=\t] (\j);};
                        \foreach \i/\j/\t in {
						v4/v1/0,
                        v1/v2/0,
						v2/v3/0
					}{\path[draw, line width=0.8, green] (\i) edge[bend left=\t] (\j);}
				\end{tikzpicture}\caption*{$(v_1,v_3)$ in $S_4$}\label{4}\end{minipage}}
	\end{figure}

\begin{proof}
    Since $S_4$ is vertex-transitive, we only need to check the following four vertex pairs. 
  \begin{description}
      \item[$(v_1,v_1)$]:
     $v_1$-out-branching: $\{v_1 v_2, v_2 v_3,v_3 v_4\}$; $v_1$-in-branching: $\{v_2 v_4, v_4 v_1, v_3 v_1\}$.
\item [$(v_1,v_2)$]: $v_1$-out-branching: $\{v_1 v_3,v_3 v_4,v_4 v_2$\}; $v_2$-in-branching: $\{v_3 v_1,v_4 v_1,v_1 v_2\}$.

    \item[$(v_2,v_1)$]: $v_2$-out-branching: $\{v_2 v_4,v_4 v_1,v_1 v_3\}$; $v_1$-in-branching: $\{v_4 v_2,v_2 v_3,v_3 v_1\}$.

    \item[$(v_1,v_3)$]: $v_1$-out-branching: $\{v_1 v_3,v_3 v_4,v_4 v_2\}$; $v_3$-in-branching: $\{v_4 v_1,v_1 v_2,v_2 v_3\}$.
  \end{description}
    So all the vertex pairs in $S_{4}$ are good vertex pairs by symmetry.
     
\end{proof}

Now we are ready to prove that each vertex pair in the five basic cases is a good vertex pair.
 
\begin{theorem}\label{thm:22}
Each vertex pair in the five basic cases is a good vertex pair.    
\end{theorem}

\begin{proof}

Suppose that $D$ is one of the basic cases above. Observe that all the five basic cases have a common subdigraph $D_1$. By replacing the $2$-path $v_4 a v_2$ with the arc $v_4v_2$ in $D_1$, we get $S_4$ as the resulting graph. And from Lemma~\ref{lem21}, we know that any vertex pair in $S_4$ is a good vertex pair.

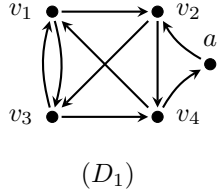
\begin{figure}[H]
        \centering
		\begin{minipage}[t]{0.19\linewidth}
			\vspace{0pt}
			\centering
			\begin{tikzpicture}[scale=0.7]
				\filldraw[black](0,0) circle (3pt)node[label=left:$v_1$](v1){};
				\filldraw[black](2,0) circle (3pt)node[label=right:$v_2$](v2){};
				\filldraw[black](0,-2) circle (3pt)node[label=left:$v_3$](v3){};
				\filldraw[black](2,-2) circle (3pt)node[label=right:$v_4$](v4){};
				
				\filldraw[black](3,-1) circle (3pt)node[label=above:$a$](a){};
				
				\foreach \i/\j/\t in {
					v1/v2/0,
					v2/v3/0,
					v3/v4/0,
					v4/v1/0,
					v1/v3/15,
                        v3/v1/15,     
					v2/v4/0,
					a/v2/15,
					v4/a/15
				}{\path[draw, line width=0.8] (\i) edge[bend left=\t] (\j);}
			\end{tikzpicture}
			\caption*{$(D_1)$}
		\end{minipage}
\end{figure}

We first claim that any vertex pair $(v_j,v_i)$ in $D$ is a good vertex pair for $i,j\in [4]$. 
By the above discussion, we can find arc-disjoint $w$-in-branching $W$ and $t$-out-branching $T$ for any vertex pair $(t,w)$ in $S_4$. If both $W$ and $T$ do not contain $v_4v_2$, then we add $a v_2$ and $v_4 a$ to $W$ and $T$ respectively, to obtain $W'$ and $T'$. It is not hard to see that $W'$ is a $w$-in-branching and $T'$ is a $t$-out-branching in $D_1$, and they are arc-disjoint. If any vertex pair $(v_j,v_i)$ in $D_1$ is a good vertex pair, certainly it is also a good vertex pair in $D$ since $D_1$ is a subdigraph of $D$. If $W$ contains the arc $v_4v_2$, then we replace $v_4v_2$ by $v_4a$ and $a v_2$ to obtain $W'$. Since $a$ has another in-neighbor $v_n$ in $D$, we can add $v_n a$ to $T$ to obtain $T'$. One can check the fact that $T'$ is a $t$-out-branching in $D$, $W'$ is a $w$-in-branching in $D$, and they are arc-disjoint. Similarly, we can discuss the case when the $t$-out-branching contains the arc $v_4v_2$. 

The vertex pair $(a,a)$ is indeed a good vertex pair because we can add $a v_2$ and $v_4 a$, respectively, to the following $v_2$-out-branching and $v_4$-in-branching in $S_4$.

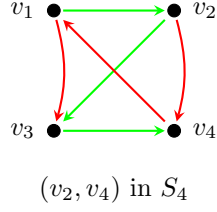
\begin{figure}[H]
		\centering
		\subfigure{\begin{minipage}[t]{0.23\linewidth}
				\centering\begin{tikzpicture}[scale=0.8]
					\filldraw[black](0,0) circle (3pt)node[label=left:$v_1$](v1){};
					\filldraw[black](2,0) circle (3pt)node[label=right:$v_2$](v2){};
					\filldraw[black](0,-2) circle (3pt)node[label=left:$v_3$](v3){};
					\filldraw[black](2,-2) circle (3pt)node[label=right:$v_4$](v4){};
					\foreach \i/\j/\t in {
						v1/v2/0,
						v2/v3/0,
						v3/v4/0
					}{\path[draw, line width=0.8, green] (\i) edge[bend left=\t] (\j);}
                        \foreach \i/\j/\t in {
						v1/v3/15,
						v4/v1/0,
						v2/v4/15
					}{\path[draw, line width=0.8, red] (\i) edge[bend left=\t] (\j);}
				\end{tikzpicture}\caption*{$(v_2,v_4)$ in $S_4$}\end{minipage}}
	\end{figure}

Now we consider the vertex pairs $(a, v_i)$ in $D$, for $i \in [4]$. Suppose that two distinct out-neighbors of $a$ are $v_2$ and $v_m$ in $D$. If we can find a good vertex pair of one of the following two types in $S_4$, then by a similar discussion above, we can transform it to a good vertex pair $(a,v_i)$ in $D$.

type 1: $(v_2,v_i)$ is a good vertex pair in $S_4$ and one of the $v_i$-in-branching does not contain $v_4v_2$;

type 2: $(v_m,v_i)$ is a good vertex pair in $S_4$ and one of the $v_i$-in-branching contains $v_4v_2$;\\
We first show how to transform such a vertex pair to a good vertex pair in $D$, and then verify every $v_i$ is in such a vertex pair for $i \in [4]$. If we have a vertex pair of type 1, then we can find a $v_i$-in-branching $W$ which does not contain $v_4v_2$ and a $v_2$-out-branching $T$ in $S_4$. We add $av_2$ and $av_m$ respectively to $W$ and $T$, and this gives us arc-disjoint $v_i$-in-branching and $a$-out-branching in $D$. Similarly, if we have a vertex pair of type 2, we replace $v_4v_2$ by $v_4a$ and $av_2$ in the $v_i$-in-branching and add $av_m$ to the $v_m$-out-branching in $S_4$. This also gives us arc-disjoint $v_i$-in-branching and $a$-out-branching in $D$. So it suffices to show any $v_i$ is in a vertex pair of type 1 or type 2. Note that $(v_2,v_4)$ is a vertex pair of type 1 since any $v_4$-in-branching could not contain $v_4v_2$ in $S_4$. And $(v_2,v_3)$ and $(v_2,v_2)$ are also vertex pairs of type 1 since we have the following structures in $S_4$:\\

\begin{figure}[H]
		\centering
		\subfigure{\begin{minipage}[t]{0.23\linewidth}
				\centering\begin{tikzpicture}[scale=0.8]
					\filldraw[black](0,0) circle (3pt)node[label=left:$v_1$](v1){};
					\filldraw[black](2,0) circle (3pt)node[label=right:$v_2$](v2){};
					\filldraw[black](0,-2) circle (3pt)node[label=left:$v_3$](v3){};
					\filldraw[black](2,-2) circle (3pt)node[label=right:$v_4$](v4){};
					\foreach \i/\j/\t in {						
						v2/v3/0,
						v3/v4/0,
						v3/v1/15					
					}{\path[draw, line width=0.8, red] (\i) edge[bend left=\t] (\j);}
                        \foreach \i/\j/\t in {
						v2/v4/15,
						v4/v1/0,
						v1/v3/15
					}{\path[draw, line width=0.8, green] (\i) edge[bend left=\t] (\j);}
				\end{tikzpicture}\caption*{$(v_2,v_3)$ in $S_4$}\end{minipage}}
		\subfigure{\begin{minipage}[t]{0.23\linewidth}
				\centering\begin{tikzpicture}[scale=0.8]
					\filldraw[black](0,0) circle (3pt)node[label=left:$v_1$](v1){};
					\filldraw[black](2,0) circle (3pt)node[label=right:$v_2$](v2){};
					\filldraw[black](0,-2) circle (3pt)node[label=left:$v_3$](v3){};
					\filldraw[black](2,-2) circle (3pt)node[label=right:$v_4$](v4){};
					\foreach \i/\j/\t in {
						v2/v4/15,		
						v2/v3/0,
						v3/v1/15
					}{\path[draw, line width=0.8, red] (\i) edge[bend left=\t] (\j);}
                        \foreach \i/\j/\t in {						
						v3/v4/0,						
						v4/v1/0,
						v1/v2/0					
					}{\path[draw, line width=0.8, green] (\i) edge[bend left=\t] (\j);}	
				\end{tikzpicture}\caption*{$(v_2,v_2)$ in $S_4$}\end{minipage}}
	\end{figure}

We now show $(a,v_1)$ is a good vertex pair in $D$, and it suffices to show $(v_m,v_1)$ is a vertex pair of type 2. Note that we have $v_m\in \{v_1,v_3,v_4\}$ in the five basic cases, so we only need to check whether $(v_1,v_1)$, $(v_3,v_1)$ and $(v_4,v_1)$ are all of type 2. In fact, we have the following structures in $S_4$, which means that the three vertex pairs are of type 2:\\

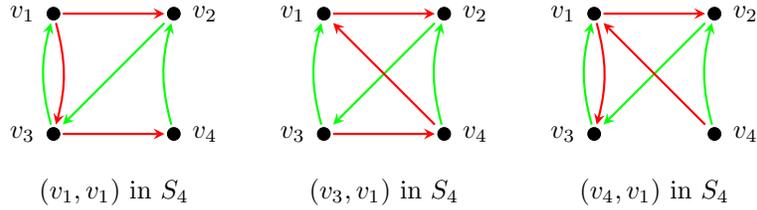
\begin{figure}[H]
         \subfigure{\begin{minipage}[t]{0.23\linewidth}
				\centering\begin{tikzpicture}[scale=0.8]
					\filldraw[black](0,0) circle (3pt)node[label=left:$v_1$](v1){};
					\filldraw[black](2,0) circle (3pt)node[label=right:$v_2$](v2){};
					\filldraw[black](0,-2) circle (3pt)node[label=left:$v_3$](v3){};
					\filldraw[black](2,-2) circle (3pt)node[label=right:$v_4$](v4){};
					\foreach \i/\j/\t in {
						v3/v1/15,
						v4/v2/15,
						v2/v3/0
					}{\path[draw, line width=0.8, green] (\i) edge[bend left=\t] (\j);}
                        \foreach \i/\j/\t in {
						v1/v2/0,						
						v3/v4/0,						
						v1/v3/15
					}{\path[draw, line width=0.8, red] (\i) edge[bend left=\t] (\j);}
				\end{tikzpicture}\caption*{$(v_1,v_1)$ in $S_4$}\end{minipage}}
        \subfigure{\begin{minipage}[t]{0.23\linewidth}
				\centering\begin{tikzpicture}[scale=0.8]
					\filldraw[black](0,0) circle (3pt)node[label=left:$v_1$](v1){};
					\filldraw[black](2,0) circle (3pt)node[label=right:$v_2$](v2){};
					\filldraw[black](0,-2) circle (3pt)node[label=left:$v_3$](v3){};
					\filldraw[black](2,-2) circle (3pt)node[label=right:$v_4$](v4){};
					\foreach \i/\j/\t in {						
						v2/v3/0,
						v3/v1/15,
						v4/v2/15
					}{\path[draw, line width=0.8, green] (\i) edge[bend left=\t] (\j);}
                        \foreach \i/\j/\t in {
						v1/v2/0,
						v3/v4/0,
						v4/v1/0
					}{\path[draw, line width=0.8, red] (\i) edge[bend left=\t] (\j);}
				\end{tikzpicture}\caption*{$(v_3,v_1)$ in $S_4$}\end{minipage}}
               \centering
		\subfigure{\begin{minipage}[t]{0.23\linewidth}
				\centering\begin{tikzpicture}[scale=0.8]
					\filldraw[black](0,0) circle (3pt)node[label=left:$v_1$](v1){};
					\filldraw[black](2,0) circle (3pt)node[label=right:$v_2$](v2){};
					\filldraw[black](0,-2) circle (3pt)node[label=left:$v_3$](v3){};
					\filldraw[black](2,-2) circle (3pt)node[label=right:$v_4$](v4){};
					\foreach \i/\j/\t in {			
						v2/v3/0,					
						v3/v1/15,
						v4/v2/15
					}{\path[draw, line width=0.8, green] (\i) edge[bend left=\t] (\j);}
                        \foreach \i/\j/\t in {
						v1/v2/0,
						v4/v1/0,
						v1/v3/15
					}{\path[draw, line width=0.8, red] (\i) edge[bend left=\t] (\j);}
				\end{tikzpicture}\caption*{$(v_4,v_1)$ in $S_4$}\end{minipage}}
	\end{figure}

As for the vertex pairs $(v_i,a)$ in $D$, in which $i \in [4]$, we suppose two distinct in-neighbors of $a$ are $v_4$ and $v_n$ in $D$. Similarly, we consider the following two types of vertex pairs:

type $1'$: $(v_i,v_4)$ is a good vertex pair in $S_4$ and one of the $v_i$-out-branching does not contain $v_4v_2$;

type $2'$: $(v_i,v_n)$ is a good vertex pair in $S_4$ and one of the $v_i$-out-branching contains $v_4v_2$.\\
We list the required structures below and omit the discussions here since it is very similar to those for vertex pairs $(a,v_i)$ in $D$.

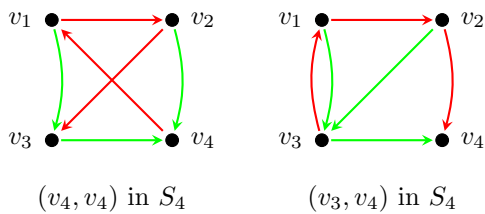
\begin{figure}[H]
		\centering
		\subfigure{\begin{minipage}[t]{0.23\linewidth}
				\centering\begin{tikzpicture}[scale=0.8]
					\filldraw[black](0,0) circle (3pt)node[label=left:$v_1$](v1){};
					\filldraw[black](2,0) circle (3pt)node[label=right:$v_2$](v2){};
					\filldraw[black](0,-2) circle (3pt)node[label=left:$v_3$](v3){};
					\filldraw[black](2,-2) circle (3pt)node[label=right:$v_4$](v4){};
					\foreach \i/\j/\t in {
						v4/v1/0,
						v1/v2/0,
                        v2/v3/0
					}{\path[draw, line width=0.8, red] (\i) edge[bend left=\t] (\j);}
                        \foreach \i/\j/\t in {
						v1/v3/15,
						v3/v4/0,
						v2/v4/15
					}{\path[draw, line width=0.8, green] (\i) edge[bend left=\t] (\j);}
				\end{tikzpicture}\caption*{$(v_4,v_4)$ in $S_4$}\end{minipage}}
		\subfigure{\begin{minipage}[t]{0.23\linewidth}
				\centering\begin{tikzpicture}[scale=0.8]
					\filldraw[black](0,0) circle (3pt)node[label=left:$v_1$](v1){};
					\filldraw[black](2,0) circle (3pt)node[label=right:$v_2$](v2){};
					\filldraw[black](0,-2) circle (3pt)node[label=left:$v_3$](v3){};
					\filldraw[black](2,-2) circle (3pt)node[label=right:$v_4$](v4){};
					\foreach \i/\j/\t in {
						v1/v2/0,
						v2/v4/15,
						v3/v1/15
					}{\path[draw, line width=0.8, red] (\i) edge[bend left=\t] (\j);}
                        \foreach \i/\j/\t in {
						v2/v3/0,
						v3/v4/0,
						v1/v3/15
					}{\path[draw, line width=0.8, green] (\i) edge[bend left=\t] (\j);}
				\end{tikzpicture}\caption*{$(v_3,v_4)$ in $S_4$}\end{minipage}}
	\end{figure}

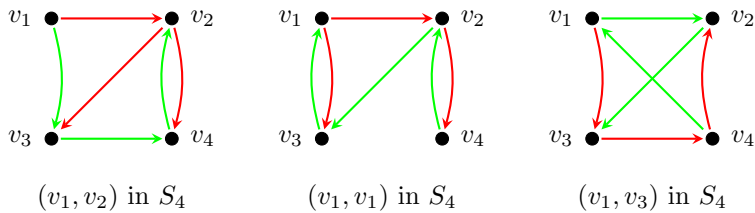
\begin{figure}[H]
		\centering
		\subfigure{\begin{minipage}[t]{0.23\linewidth}
				\centering\begin{tikzpicture}[scale=0.8]
					\filldraw[black](0,0) circle (3pt)node[label=left:$v_1$](v1){};
					\filldraw[black](2,0) circle (3pt)node[label=right:$v_2$](v2){};
					\filldraw[black](0,-2) circle (3pt)node[label=left:$v_3$](v3){};
					\filldraw[black](2,-2) circle (3pt)node[label=right:$v_4$](v4){};
					\foreach \i/\j/\t in {
						v1/v2/0,
						v2/v3/0,
						v2/v4/15
					}{\path[draw, line width=0.8, red] (\i) edge[bend left=\t] (\j);}
                        \foreach \i/\j/\t in {
						v3/v4/0,
						v1/v3/15,
						v4/v2/15
					}{\path[draw, line width=0.8, green] (\i) edge[bend left=\t] (\j);}
                \end{tikzpicture}\caption*{$(v_1,v_2)$ in $S_4$}\end{minipage}}
		\subfigure{\begin{minipage}[t]{0.23\linewidth}
				\centering\begin{tikzpicture}[scale=0.8]
					\filldraw[black](0,0) circle (3pt)node[label=left:$v_1$](v1){};
					\filldraw[black](2,0) circle (3pt)node[label=right:$v_2$](v2){};
					\filldraw[black](0,-2) circle (3pt)node[label=left:$v_3$](v3){};
					\filldraw[black](2,-2) circle (3pt)node[label=right:$v_4$](v4){};
					\foreach \i/\j/\t in {
						v2/v3/0,
						v3/v1/15,
						v4/v2/15
					}{\path[draw, line width=0.8, green] (\i) edge[bend left=\t] (\j);}		
				
                        \foreach \i/\j/\t in {
						v1/v2/0,
						v1/v3/15,
						v2/v4/15
					}{\path[draw, line width=0.8, red] (\i) edge[bend left=\t] (\j);}
                \end{tikzpicture}\caption*{$(v_1,v_1)$ in $S_4$}\end{minipage}}
        \subfigure{\begin{minipage}[t]{0.23\linewidth}
				\centering\begin{tikzpicture}[scale=0.8]
					\filldraw[black](0,0) circle (3pt)node[label=left:$v_1$](v1){};
					\filldraw[black](2,0) circle (3pt)node[label=right:$v_2$](v2){};
					\filldraw[black](0,-2) circle (3pt)node[label=left:$v_3$](v3){};
					\filldraw[black](2,-2) circle (3pt)node[label=right:$v_4$](v4){};
					\foreach \i/\j/\t in {
						v3/v4/0,
						v1/v3/15,
						v4/v2/15
					}{\path[draw, line width=0.8, red] (\i) edge[bend left=\t] (\j);}
				
                        \foreach \i/\j/\t in {
						v1/v2/0,
						v2/v3/0,
						v4/v1/0
					}{\path[draw, line width=0.8, green] (\i) edge[bend left=\t] (\j);}
                \end{tikzpicture}\caption*{$(v_1,v_3)$ in $S_4$}\end{minipage}}
	\end{figure}
\end{proof}

\subsection{Combinations of the Five Basic Cases in Appendix}

By deleting $v_3 v_1$ in the five basic cases, we get the following structures $(i)-(v)$.\\
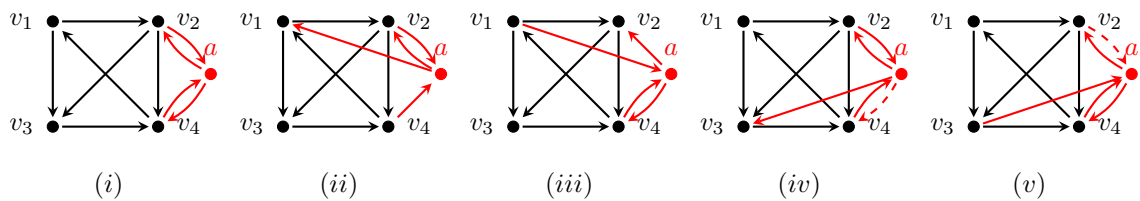
\begin{figure}[H]
		\begin{minipage}[t]{0.19\linewidth}
			\vspace{0pt}
			\centering
			\begin{tikzpicture}[scale=0.7]
				\filldraw[black](0,0) circle (3pt)node[label=left:$v_1$](v1){};
				\filldraw[black](2,0) circle (3pt)node[label=right:$v_2$](v2){};
				\filldraw[black](0,-2) circle (3pt)node[label=left:$v_3$](v3){};
				\filldraw[black](2,-2) circle (3pt)node[label=right:$v_4$](v4){};
				
				\filldraw[red](3,-1) circle (3pt)node[label=above:$a$](a){};
				
				\foreach \i/\j/\t in {
					v1/v2/0,
					v2/v3/0,
					v3/v4/0,
					v4/v1/0,
                        v1/v3/0,     
					v2/v4/0
				}{\path[draw, line width=0.8] (\i) edge[bend left=\t] (\j);}
                \foreach \i/\j/\t in {
					a/v4/15,
					a/v2/15,
					v4/a/15,
					v2/a/15
				}{\path[draw, red, line width=0.8] (\i) edge[bend left=\t] (\j);}
			\end{tikzpicture}
			\caption*{$(i)$}
		\end{minipage}
		\hfill
		\begin{minipage}[t]{0.19\linewidth}
			\vspace{0pt}
			\centering
			\begin{tikzpicture}[scale=0.7]
				\filldraw[black](0,0) circle (3pt)node[label=left:$v_1$](v1){};
				\filldraw[black](2,0) circle (3pt)node[label=right:$v_2$](v2){};
				\filldraw[black](0,-2) circle (3pt)node[label=left:$v_3$](v3){};
				\filldraw[black](2,-2) circle (3pt)node[label=right:$v_4$](v4){};
				
				\filldraw[red](3,-1) circle (3pt)node[label=above:$a$](a){};
				
				\foreach \i/\j/\t in {
					v1/v2/0,
					v2/v3/0,
					v3/v4/0,
					v4/v1/0,
					v1/v3/0,
					v2/v4/0
				}{\path[draw, line width=0.8] (\i) edge[bend left=\t] (\j);}
				
				\foreach \i/\j/\t in {
					a/v2/15,
					v2/a/15,
					v4/a/0,
					a/v1/0
				}{\path[draw, red, line width=0.8] (\i) edge[bend left=\t] (\j);}
			\end{tikzpicture}
			\caption*{$(ii)$}
		\end{minipage}
		\hfill
		\begin{minipage}[t]{0.19\linewidth}
			\vspace{0pt}
			\centering
			\begin{tikzpicture}[scale=0.7]
				\filldraw[black](0,0) circle (3pt)node[label=left:$v_1$](v1){};
				\filldraw[black](2,0) circle (3pt)node[label=right:$v_2$](v2){};
				\filldraw[black](0,-2) circle (3pt)node[label=left:$v_3$](v3){};
				\filldraw[black](2,-2) circle (3pt)node[label=right:$v_4$](v4){};
				
				\filldraw[red](3,-1) circle (3pt)node[label=above:$a$](a){};
				
				\foreach \i/\j/\t in {
					v1/v2/0,
					v2/v3/0,
					v3/v4/0,
					v4/v1/0,
					v1/v3/0,
					v2/v4/0
				}{\path[draw, line width=0.8] (\i) edge[bend left=\t] (\j);}
				
				\foreach \i/\j/\t in {
					a/v2/0,
					a/v4/15,
					v4/a/15,
					v1/a/0
				}{\path[draw, red, line width=0.8] (\i) edge[bend left=\t] (\j);}
			\end{tikzpicture}
			\caption*{$(iii)$}
		\end{minipage}
		\hfill
		\begin{minipage}[t]{0.19\linewidth}
			\vspace{0pt}
			\centering
			\begin{tikzpicture}[scale=0.7]
				\filldraw[black](0,0) circle (3pt)node[label=left:$v_1$](v1){};
				\filldraw[black](2,0) circle (3pt)node[label=right:$v_2$](v2){};
				\filldraw[black](0,-2) circle (3pt)node[label=left:$v_3$](v3){};
				\filldraw[black](2,-2) circle (3pt)node[label=right:$v_4$](v4){};
				
				\filldraw[red](3,-1) circle (3pt)node[label=above:$a$](a){};
				
				\foreach \i/\j/\t in {
					v1/v2/0,
					v2/v3/0,
					v3/v4/0,
					v4/v1/0,
					v1/v3/0,
					v2/v4/0
				}{\path[draw, line width=0.8] (\i) edge[bend left=\t] (\j);}
				
				\foreach \i/\j/\t in {
					a/v2/15,
					v2/a/15,
					v4/a/15,
					a/v3/0
				}{\path[draw, red, line width=0.8] (\i) edge[bend left=\t] (\j);}
				
				\path[draw, dashed, red, line width=0.8] (a) edge[bend left=15] (v4);
			\end{tikzpicture}
			\caption*{$(iv)$}
		\end{minipage}
		\hfill
		\begin{minipage}[t]{0.19\linewidth}
			\vspace{0pt}
			\centering
			\begin{tikzpicture}[scale=0.7]
				\filldraw[black](0,0) circle (3pt)node[label=left:$v_1$](v1){};
				\filldraw[black](2,0) circle (3pt)node[label=right:$v_2$](v2){};
				\filldraw[black](0,-2) circle (3pt)node[label=left:$v_3$](v3){};
				\filldraw[black](2,-2) circle (3pt)node[label=right:$v_4$](v4){};
				
				\filldraw[red](3,-1) circle (3pt)node[label=above:$a$](a){};
				
				\foreach \i/\j/\t in {
					v1/v2/0,
					v2/v3/0,
					v3/v4/0,
					v4/v1/0,
					v1/v3/0,
					v2/v4/0
				}{\path[draw, line width=0.8] (\i) edge[bend left=\t] (\j);}
				
				\foreach \i/\j/\t in {
					a/v2/15,
					a/v4/15,
					v4/a/15,
					v3/a/0
				}{\path[draw, red, line width=0.8] (\i) edge[bend left=\t] (\j);}
				
				\path[draw, dashed, red, line width=0.8] (v2) edge[bend left=15] (a);
			\end{tikzpicture}
			\caption*{$(v)$}
		\end{minipage}
	\end{figure}
\noindent Reverse all the arcs in $(i)-(v)$, rotate 180 degrees clockwise, and relabel, we can obtain the corresponding reversed and rotated structures $(i)^*-(v)^*$.

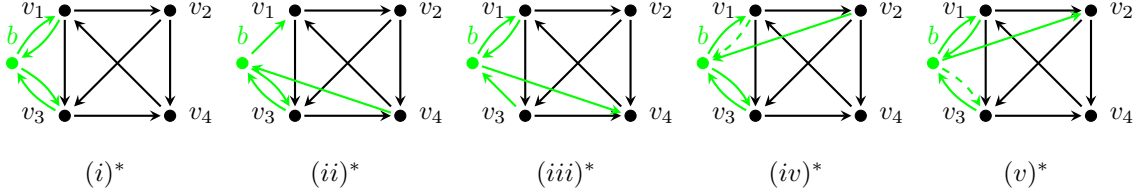
	\begin{figure}[H]
		\begin{minipage}[t]{0.19\linewidth}
			\vspace{0pt}
			\centering
			\begin{tikzpicture}[scale=0.7]
				\filldraw[black](0,0) circle (3pt)node[label=left:$v_1$](v1){};
				\filldraw[black](2,0) circle (3pt)node[label=right:$v_2$](v2){};
				\filldraw[black](0,-2) circle (3pt)node[label=left:$v_3$](v3){};
				\filldraw[black](2,-2) circle (3pt)node[label=right:$v_4$](v4){};
				\filldraw[green](-1,-1) circle (3pt)node[label=above:$b$](b){};
				
				\foreach \i/\j/\t in {
					v1/v2/0,
					v2/v3/0,
					v3/v4/0,
					v4/v1/0,
					v1/v3/0,
					v2/v4/0
				}{\path[draw, line width=0.8] (\i) edge[bend left=\t] (\j);}
				
				\foreach \i/\j/\t in {
					b/v1/15,
					b/v3/15,
					v3/b/15,
					v1/b/15
				}{\path[draw, green, line width=0.8] (\i) edge[bend left=\t] (\j);}
			\end{tikzpicture}
			\caption*{$(i)^*$}
		\end{minipage}
		\hfill
		\begin{minipage}[t]{0.19\linewidth}
			\vspace{0pt}
			\centering
			\begin{tikzpicture}[scale=0.7]
				\filldraw[black](0,0) circle (3pt)node[label=left:$v_1$](v1){};
				\filldraw[black](2,0) circle (3pt)node[label=right:$v_2$](v2){};
				\filldraw[black](0,-2) circle (3pt)node[label=left:$v_3$](v3){};
				\filldraw[black](2,-2) circle (3pt)node[label=right:$v_4$](v4){};
				\filldraw[green](-1,-1) circle (3pt)node[label=above:$b$](b){};
				
				\foreach \i/\j/\t in {
					v1/v2/0,
					v2/v3/0,
					v3/v4/0,
					v4/v1/0,
					v1/v3/0,
					v2/v4/0
				}{\path[draw, line width=0.8] (\i) edge[bend left=\t] (\j);}
				
				\foreach \i/\j/\t in {
					b/v1/0,
					b/v3/15,
					v3/b/15,
					v4/b/0
				}{\path[draw, green, line width=0.8] (\i) edge[bend left=\t] (\j);}
			\end{tikzpicture}
			\caption*{$(ii)^*$}
		\end{minipage}
		\hfill
		\begin{minipage}[t]{0.19\linewidth}
			\vspace{0pt}
			\centering
			\begin{tikzpicture}[scale=0.7]
				\filldraw[black](0,0) circle (3pt)node[label=left:$v_1$](v1){};
				\filldraw[black](2,0) circle (3pt)node[label=right:$v_2$](v2){};
				\filldraw[black](0,-2) circle (3pt)node[label=left:$v_3$](v3){};
				\filldraw[black](2,-2) circle (3pt)node[label=right:$v_4$](v4){};
				\filldraw[green](-1,-1) circle (3pt)node[label=above:$b$](b){};
				
				\foreach \i/\j/\t in {
					v1/v2/0,
					v2/v3/0,
					v3/v4/0,
					v4/v1/0,
					v1/v3/0,
					v2/v4/0
				}{\path[draw, line width=0.8] (\i) edge[bend left=\t] (\j);}
				
				\foreach \i/\j/\t in {
					b/v1/15,
					b/v4/0,
					v1/b/15,
					v3/b/0
				}{\path[draw, green, line width=0.8] (\i) edge[bend left=\t] (\j);}
			\end{tikzpicture}
			\caption*{$(iii)^*$}
		\end{minipage}
		\hfill
		\begin{minipage}[t]{0.19\linewidth}
			\vspace{0pt}
			\centering
			\begin{tikzpicture}[scale=0.7]
				\filldraw[black](0,0) circle (3pt)node[label=left:$v_1$](v1){};
				\filldraw[black](2,0) circle (3pt)node[label=right:$v_2$](v2){};
				\filldraw[black](0,-2) circle (3pt)node[label=left:$v_3$](v3){};
				\filldraw[black](2,-2) circle (3pt)node[label=right:$v_4$](v4){};
				\filldraw[green](-1,-1) circle (3pt)node[label=above:$b$](b){};
				
				\foreach \i/\j/\t in {
					v1/v2/0,
					v2/v3/0,
					v3/v4/0,
					v4/v1/0,
					v1/v3/0,
					v2/v4/0
				}{\path[draw, line width=0.8] (\i) edge[bend left=\t] (\j);}
				
				\foreach \i/\j/\t in {
					b/v1/15,
					b/v3/15,
					v3/b/15,
					v2/b/0
				}{\path[draw, green, line width=0.8] (\i) edge[bend left=\t] (\j);}
				
				\path[draw, dashed, green, line width=0.8] (v1) edge[bend left=15] (b);
			\end{tikzpicture}
			\caption*{$(iv)^*$}
		\end{minipage}
		\hfill
		\begin{minipage}[t]{0.19\linewidth}
			\vspace{0pt}
			\centering
			\begin{tikzpicture}[scale=0.7]
				\filldraw[black](0,0) circle (3pt)node[label=left:$v_1$](v1){};
				\filldraw[black](2,0) circle (3pt)node[label=right:$v_2$](v2){};
				\filldraw[black](0,-2) circle (3pt)node[label=left:$v_3$](v3){};
				\filldraw[black](2,-2) circle (3pt)node[label=right:$v_4$](v4){};
				\filldraw[green](-1,-1) circle (3pt)node[label=above:$b$](b){};
				
				\foreach \i/\j/\t in {
					v1/v2/0,
					v2/v3/0,
					v3/v4/0,
					v4/v1/0,
					v1/v3/0,
					v2/v4/0
				}{\path[draw, line width=0.8] (\i) edge[bend left=\t] (\j);}
				
				\foreach \i/\j/\t in {
					b/v1/15,
					v1/b/15,
					v3/b/15,
					b/v2/0
				}{\path[draw, green, line width=0.8] (\i) edge[bend left=\t] (\j);}
				
				\path[draw, dashed, green, line width=0.8] (b) edge[bend left=15] (v3);
			\end{tikzpicture}
			\caption*{$(v)^*$}
		\end{minipage}
	\end{figure}
 
\noindent We say $D$ is a combination of $(e)$ and $(f)^*$, in which $e,f \in \{i,ii,iii,iv,v\}$, if $D$ is the digraph with $V(D)=\{a,b,v_1,v_2,v_3,v_4\}$ and $A(D) = A((e)) \cup A((f)^*)$, and we denote it by $(e)\times(f)^*$. 

Note that all left cases in Appendix of~\cite{ai2024} can be represented as some $(e)\times(f)^*$, we only need to prove the following theorem to complete the proof for all cases in Appendix of~\cite{ai2024}.

\begin{theorem}
    Each vertex pair in $(e)\times (f)^*$ is a good vertex pair, in which $e,f \in \{i,ii,iii,iv,v\}$.
\end{theorem} 

\begin{proof}
    
Suppose $D$ is the combination of $(e)$ and $(f)^*$.

We first replace the $2$-path $v_3 b v_1$  with $v_3 v_1$ and remove all other arcs that are incident to $b$  in $D$ to obtain a new graph $D'$. By Theorem~\ref{thm:22}, we can find arc-disjoint $w$-in-branching $W$ and $t$-out-branching $T$ for any vertex pair $(t,w)$ in $D'$. If $W$ contains the arc $v_3v_1$, then we replace $v_3v_1$ by $v_3b$ and $b v_1$ to obtain $W'$, observe that $W'$ is a $w$-in-branching in $D$. Since $b$ has another in-neighbor $v_n$ in $D$, we add $v_n b$ to $T$ to obtain $T'$, and obviously  $T'$ is a $t$-out-branching in $D$ and it is arc-disjoint with $W'$. Similarly, we can discuss the case when $t$-out-branching contains the arc $v_3v_1$. So we may assume that $W$ and $T$ do not contain $v_3v_1$. In this case, we can simply add $b v_1$ and $v_3b$ to $W$ and $T$ respectively to obtain a $w$-in-branching and a $t$-out-branching in $D$, and they are arc-disjoint. This means that any vertex pair $(t,w)$ in which $w,t \in \{v_1,v_2,v_3,v_4,a\}$ is a good vertex pair in $D$.

Similarly, we replace the $2$-path $v_4 a v_2$, remove all other arcs that are incident to $a$ in $D$, and reverse all the arcs to obtain a new graph $D'$. Note that reversing all the arcs would only change a good vertex pair $(t,w)$ in the original graph into a good vertex pair $(w,t)$ in the resulting graph, so by analogous discussion as above, we deduce that any vertex pair $(t,w)$ in which $w,t \in \{v_1,v_2,v_3,v_4,b\}$ is a good vertex pair in $D$.

Now the only vertex pairs we need to consider in the following part are $(a,b)$ and $(b,a)$. If $(a,b)$ (resp., $(b,a)$) is a good vertex pair in $(e)\times (f)^*$, then $(a,b)$ (resp., $(b,a)$) is also a good vertex pair in $(f)\times (e)^*$ since $(e)\times (f)^*$ can be transformed into $(f)\times (e)^*$ by reversing all the arcs and relabeling, where $e$ and $f$ are elements of $\{i,ii,iii,iv,v\}$. 

We suppose $a$ (resp., $b$) has two out-neighbors $v_2,v_{m_1}$ (resp.,$v_1$,$v_{m_2}$) and two in-neighbors $v_4,v_{n_1}$ (resp., $v_3$,$v_{n_2}$). We can find the following 4 structures in $S_4$, in which each can be transformed into the desired in-branching and out-branching in $D$. We use Structure 1 to show how to achieve this. Note that $a v_2$ is in every $D$, so if $v_2 b$ is in $D$ (namely, $D$ contains $(iv)^*$), we can add $a v_2$ to the $v_2$-out-branching and replace $v_3 v_1$ by $v_3 b,b v_1$ to obtain an $a$-out-branching in $D$. Then by adding $v_2 b, a v_{m_1}$ to the $v_2$-in-branching, we can get a $b$-in-branching in $D$. So we can always find arc-disjoint $b$-in-branching and $a$-out-branching in any $D$ that contains $(iv)^*$ or $(iv)$. Namely, $(a,b)$ would be a good vertex pair in such $D$. Similarly, we have: Structure 2 implies $(a,b)$ is a good vertex pair in $D$ that contains $(ii)$ and $(b,a)$ is a good vertex pair in $D$ that contains $(v)$, and Structure 3 implies $(b,a)$ is a good  vertex pair in $D$ that contains $(iii)$. Moreover, we can find Structure 4 in any $D$ that contains $(i)$, which implies $(b,a)$ is a good vertex pair in such $D$.

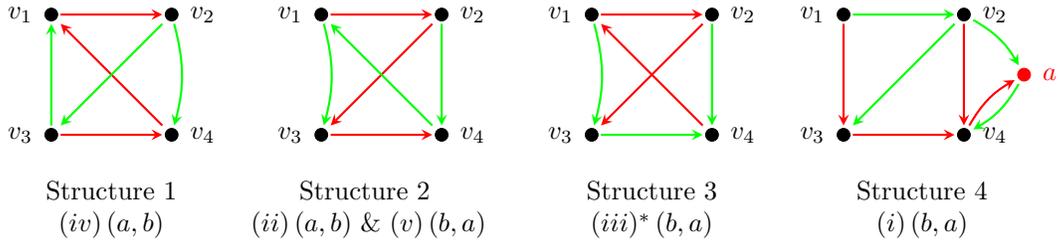
\begin{figure}[H]

 \subfigure{\begin{minipage}[t]{0.23\linewidth}
				\centering\begin{tikzpicture}[scale=0.8]
					\filldraw[black](0,0) circle (3pt)node[label=left:$v_1$](v1){};
					\filldraw[black](2,0) circle (3pt)node[label=right:$v_2$](v2){};
					\filldraw[black](0,-2) circle (3pt)node[label=left:$v_3$](v3){};
					\filldraw[black](2,-2) circle (3pt)node[label=right:$v_4$](v4){};
					\foreach \i/\j/\t in {
						v1/v2/0,
						v4/v1/0,
						v3/v4/0
					}{\path[draw, line width=0.8, red] (\i) edge[bend left=\t] (\j);}
                        \foreach \i/\j/\t in {
						v3/v1/0,
						v2/v3/0,
						v2/v4/15
					}{\path[draw, line width=0.8, green] (\i) edge[bend left=\t] (\j);}
				\end{tikzpicture}\caption*{Structure 1\\$(iv)\,(a,b)$}\end{minipage}}
		\subfigure{\begin{minipage}[t]{0.23\linewidth}
				\centering\begin{tikzpicture}[scale=0.8]
					\filldraw[black](0,0) circle (3pt)node[label=left:$v_1$](v1){};
					\filldraw[black](2,0) circle (3pt)node[label=right:$v_2$](v2){};
					\filldraw[black](0,-2) circle (3pt)node[label=left:$v_3$](v3){};
					\filldraw[black](2,-2) circle (3pt)node[label=right:$v_4$](v4){};
					\foreach \i/\j/\t in {
						v1/v2/0,
						v2/v3/0,
						v3/v4/0
					}{\path[draw, line width=0.8, red] (\i) edge[bend left=\t] (\j);}
                        \foreach \i/\j/\t in {
						v2/v4/0,
                        v4/v1/0,
						v1/v3/15
					}{\path[draw, line width=0.8, green] (\i) edge[bend left=\t] (\j);}
				\end{tikzpicture}\caption*{$\quad \,\,\,$ Structure 2\\$(ii)\,(a,b)$ $\&$ $(v)\, (b,a)$}\end{minipage}}
         \subfigure{\begin{minipage}[t]{0.23\linewidth}
				\centering\begin{tikzpicture}[scale=0.8]
					\filldraw[black](0,0) circle (3pt)node[label=left:$v_1$](v1){};
					\filldraw[black](2,0) circle (3pt)node[label=right:$v_2$](v2){};
					\filldraw[black](0,-2) circle (3pt)node[label=left:$v_3$](v3){};
					\filldraw[black](2,-2) circle (3pt)node[label=right:$v_4$](v4){};
					\foreach \i/\j/\t in {
						v1/v2/0,
						v2/v3/0,
						v4/v1/0
					}{\path[draw, line width=0.8, red] (\i) edge[bend left=\t] (\j);}
                        \foreach \i/\j/\t in {
						v1/v3/15,
                        v3/v4/0,
						v2/v4/0
					}{\path[draw, line width=0.8, green] (\i) edge[bend left=\t] (\j);}
				\end{tikzpicture}\caption*{Structure 3\\$(iii)^*\, (b,a)$}\end{minipage}}
         \subfigure{\begin{minipage}[t]{0.23\linewidth}
				\centering\begin{tikzpicture}[scale=0.8]
					\filldraw[black](0,0) circle (3pt)node[label=left:$v_1$](v1){};
					\filldraw[black](2,0) circle (3pt)node[label=right:$v_2$](v2){};
					\filldraw[black](0,-2) circle (3pt)node[label=left:$v_3$](v3){};
					\filldraw[black](2,-2) circle (3pt)node[label=right:$v_4$](v4){};
                        \filldraw[red](3,-1) circle 
                    (3pt)node[label=right:$a$](a){};
                    
					\foreach \i/\j/\t in {
						v1/v3/0,
						v2/v4/0,
						v3/v4/0,
                            v4/a/15
					}{\path[draw, line width=0.8, red] (\i) edge[bend left=\t] (\j);}
                        \foreach \i/\j/\t in {
						a/v4/15,
                            v2/a/15,
                            v2/v3/0,
						v1/v2/0
					}{\path[draw, line width=0.8, green] (\i) edge[bend left=\t] (\j);}
		\end{tikzpicture}\caption*{Structure 4\\$(i)\,(b,a)$}\end{minipage}}  
    
\end{figure}

So we only need to check the following vertex pairs (total 9 distinct vertex pairs): $(b,a)$ of $(e)\times(f)^*$ in which $e,f\in\{ii,iv\}$ and $(a,b)$ of $(e)\times (f)^*$ in which $e,f\in \{i,iii,v\}$. 

For $(b,a)$, we have the following results:

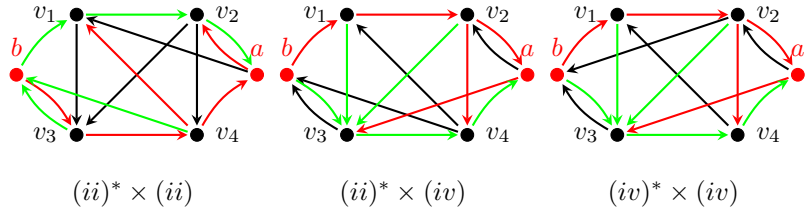
\begin{figure}[H]
   \centering
		\subfigure{\begin{minipage}[t]{0.23\linewidth}
			\vspace{0pt}
			\centering
			\begin{tikzpicture}[scale=0.8]
				\filldraw[black](0,0) circle (3pt)node[label=left:$v_1$](v1){};
				\filldraw[black](2,0) circle (3pt)node[label=right:$v_2$](v2){};
				\filldraw[black](0,-2) circle (3pt)node[label=left:$v_3$](v3){};
				\filldraw[black](2,-2) circle (3pt)node[label=right:$v_4$](v4){};
				\filldraw[red](-1,-1) circle (3pt)node[label=above:$b$](b){};
				\filldraw[red](3,-1) circle (3pt)node[label=above:$a$](a){};
				\foreach \i/\j/\t in {
					v2/v4/0,
					v2/v3/0,
					v1/v3/0,
                         a/v1/0                         
				}{\path[draw, line width=0.8] (\i) edge[bend left=\t] (\j);}
				
				\foreach \i/\j/\t in {
					b/v3/15,
					v3/v4/0,
					v4/v1/0,
                        v4/a/15,
					a/v2/15
				}{\path[draw, red, line width=0.8] (\i) edge[bend left=\t] (\j);}
                \foreach \i/\j/\t in {
					v3/b/15,
					b/v1/15,
					v1/v2/0,
					v4/b/0,
                        v2/a/15
				}{\path[draw, green, line width=0.8] (\i) edge[bend left=\t] (\j);}
			\end{tikzpicture}
			\caption*{$(ii)^*\times(ii)$}
		\end{minipage}}	
	\subfigure{\begin{minipage}[t]{0.23\linewidth}
			\vspace{0pt}
			\centering
			\begin{tikzpicture}[scale=0.8]
				\filldraw[black](0,0) circle (3pt)node[label=left:$v_1$](v1){};
				\filldraw[black](2,0) circle (3pt)node[label=right:$v_2$](v2){};
				\filldraw[black](0,-2) circle (3pt)node[label=left:$v_3$](v3){};
				\filldraw[black](2,-2) circle (3pt)node[label=right:$v_4$](v4){};
				
				\filldraw[red](3,-1) circle (3pt)node[label=above:$a$](a){};
				\filldraw[red](-1,-1) circle (3pt)node[label=above:$b$](b){};
				\foreach \i/\j/\t in {					
					v4/v1/0,
                    a/v2/15,    v4/b/0,
                v3/b/15
				}{\path[draw, line width=0.8] (\i) edge[bend left=\t] (\j);}
				
				\foreach \i/\j/\t in {
					b/v1/15,
				a/v3/0,	
				v1/v2/0,
                        v2/v4/0,
					v2/a/15
				}{\path[draw, red, line width=0.8] (\i) edge[bend left=\t] (\j);}
                \foreach \i/\j/\t in {
					b/v3/15,
					v3/v4/0,
					v2/v3/0,
				v1/v3/0,	v4/a/15
				}{\path[draw, green, line width=0.8] (\i) edge[bend left=\t] (\j);}
			\end{tikzpicture}
			\caption*{$(ii)^*\times(iv)$}
		\end{minipage}}
		\subfigure{\begin{minipage}[t]{0.23\linewidth}
			\vspace{0pt}
			\centering
			\begin{tikzpicture}[scale=0.8]
				\filldraw[black](0,0) circle (3pt)node[label=left:$v_1$](v1){};
				\filldraw[black](2,0) circle (3pt)node[label=right:$v_2$](v2){};
				\filldraw[black](0,-2) circle (3pt)node[label=left:$v_3$](v3){};
				\filldraw[black](2,-2) circle (3pt)node[label=right:$v_4$](v4){};
				\filldraw[red](-1,-1) circle (3pt)node[label=above:$b$](b){};
				\filldraw[red](3,-1) circle (3pt)node[label=above:$a$](a){};
				
				\foreach \i/\j/\t in {
					v2/b/0,
					v4/v1/0,			
					a/v2/15,
                        v3/b/15     
				}{\path[draw, line width=0.8] (\i) edge[bend left=\t] (\j);}
				
				\foreach \i/\j/\t in {
					b/v1/15,
					v1/v2/0,
					v2/a/15,
					v2/v4/0,
                        a/v3/0
				}{\path[draw, red, line width=0.8] (\i) edge[bend left=\t] (\j);}
                \foreach \i/\j/\t in {
					b/v3/15,
					v1/v3/0,
					v2/v3/0,
					v3/v4/0,
                        v4/a/15
				}{\path[draw, green, line width=0.8] (\i) edge[bend left=\t] (\j);}
			\end{tikzpicture}
			\caption*{$(iv)^*\times(iv)$}
		\end{minipage}}
  		
	\end{figure}

For $(a,b)$, we have the following results:
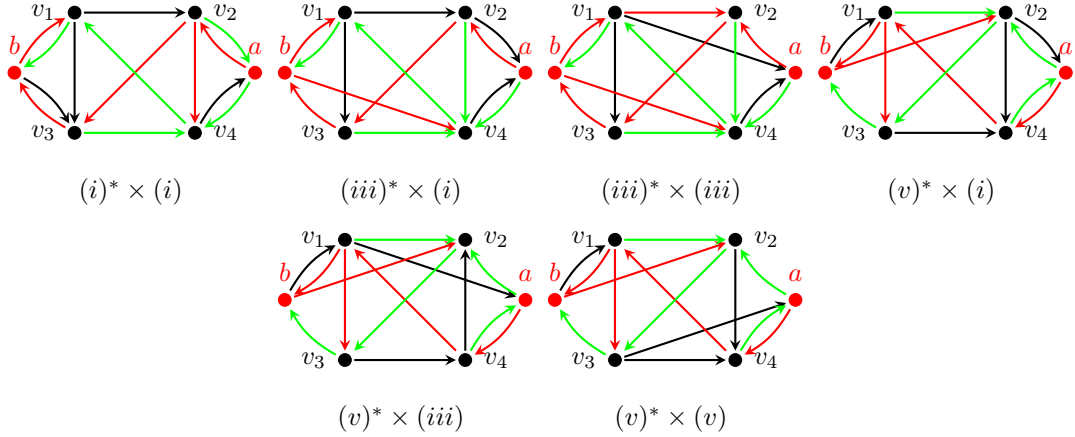
\begin{figure}[H]
\centering
       \subfigure{\begin{minipage}[t]{0.23\linewidth}
			\vspace{0pt}
			\centering
			\begin{tikzpicture}[scale=0.8]
				\filldraw[black](0,0) circle (3pt)node[label=left:$v_1$](v1){};
				\filldraw[black](2,0) circle (3pt)node[label=right:$v_2$](v2){};
				\filldraw[black](0,-2) circle (3pt)node[label=left:$v_3$](v3){};
				\filldraw[black](2,-2) circle (3pt)node[label=right:$v_4$](v4){};
				\filldraw[red](-1,-1) circle (3pt)node[label=above:$b$](b){};
				\filldraw[red](3,-1) circle (3pt)node[label=above:$a$](a){};
				
				\foreach \i/\j/\t in {
					v4/a/15,
					v1/v2/0,
					v1/v3/0,
					b/v3/15
				}{\path[draw, line width=0.8] (\i) edge[bend left=\t] (\j);}
				
				\foreach \i/\j/\t in {
					a/v2/15,
					v2/v4/0,
					v2/v3/0,
					v3/b/15,
                        b/v1/15
				}{\path[draw, red, line width=0.8] (\i) edge[bend left=\t] (\j);}
				
				\foreach \i/\j/\t in {
					v2/a/15,
					a/v4/15,
					v3/v4/0,
					v4/v1/0,
                        v1/b/15
				}{\path[draw, green, line width=0.8] (\i) edge[bend left=\t] (\j);}
			
			\end{tikzpicture}
			\caption*{$(i)^*\times(i)$}
		\end{minipage}}	
		\subfigure{\begin{minipage}[t]{0.23\linewidth}
			\vspace{0pt}
			\centering
			\begin{tikzpicture}[scale=0.8]
				\filldraw[black](0,0) circle (3pt)node[label=left:$v_1$](v1){};
				\filldraw[black](2,0) circle (3pt)node[label=right:$v_2$](v2){};
				\filldraw[black](0,-2) circle (3pt)node[label=left:$v_3$](v3){};
				\filldraw[black](2,-2) circle (3pt)node[label=right:$v_4$](v4){};
				
				\filldraw[red](3,-1) circle (3pt)node[label=above:$a$](a){};
				\filldraw[red](-1,-1) circle (3pt)node[label=above:$b$](b){};
				\foreach \i/\j/\t in {
					v1/v2/0,
					v2/a/15,
					v1/v3/0,
					v4/a/15
				}{\path[draw, line width=0.8] (\i) edge[bend left=\t] (\j);}
				
				\foreach \i/\j/\t in {
					a/v2/15,
					v2/v3/0,
					v3/b/15,
					b/v1/15,
                        b/v4/0
				}{\path[draw, red, line width=0.8] (\i) edge[bend left=\t] (\j);}
                \foreach \i/\j/\t in {
					a/v4/15,
					v2/v4/0,
					v4/v1/0,
					v3/v4/0,
                        v1/b/15
				}{\path[draw, green, line width=0.8] (\i) edge[bend left=\t] (\j);}
			\end{tikzpicture}
			\caption*{$(iii)^*\times(i)$}
		\end{minipage}}
        		\subfigure{\begin{minipage}[t]{0.23\linewidth}
			\vspace{0pt}
			\centering
			\begin{tikzpicture}[scale=0.8]
				\filldraw[black](0,0) circle (3pt)node[label=left:$v_1$](v1){};
				\filldraw[black](2,0) circle (3pt)node[label=right:$v_2$](v2){};
				\filldraw[black](0,-2) circle (3pt)node[label=left:$v_3$](v3){};
				\filldraw[black](2,-2) circle (3pt)node[label=right:$v_4$](v4){};
				\filldraw[red](-1,-1) circle (3pt)node[label=above:$b$](b){};
				\filldraw[red](3,-1) circle (3pt)node[label=above:$a$](a){};
				
				\foreach \i/\j/\t in {
					v4/a/15,
					v1/v3/0,
					v1/a/0
				}{\path[draw, line width=0.8] (\i) edge[bend left=\t] (\j);}
				
				\foreach \i/\j/\t in {
					a/v2/15,
					v2/v3/0,
					v3/b/15,
					b/v1/15,
                        v1/v2/0,
                        b/v4/0
				}{\path[draw, red, line width=0.8] (\i) edge[bend left=\t] (\j);}
				
				\foreach \i/\j/\t in {
					a/v4/15,
					v4/v1/0,
					v2/v4/0,
					v3/v4/0,
                        v1/b/15
	}{\path[draw, green, line width=0.8] (\i) edge[bend left=\t] (\j);}
			\end{tikzpicture}
			\caption*{$(iii)^*\times(iii)$}
		\end{minipage}}
		\subfigure{\begin{minipage}[t]{0.23\linewidth}
			\vspace{0pt}
			\centering
			\begin{tikzpicture}[scale=0.8]
				\filldraw[black](0,0) circle (3pt)node[label=left:$v_1$](v1){};
				\filldraw[black](2,0) circle (3pt)node[label=right:$v_2$](v2){};
				\filldraw[black](0,-2) circle (3pt)node[label=left:$v_3$](v3){};
				\filldraw[black](2,-2) circle (3pt)node[label=right:$v_4$](v4){};
				\filldraw[red](-1,-1) circle (3pt)node[label=above:$b$](b){};
				\filldraw[red](3,-1) circle (3pt)node[label=above:$a$](a){};
				\foreach \i/\j/\t in {
					v2/a/15,
					v2/v4/0,
					v3/v4/0,
					b/v1/15			
				}{\path[draw, line width=0.8] (\i) edge[bend left=\t] (\j);}
				
				\foreach \i/\j/\t in {
					a/v4/15,
					v4/v1/0,
					v1/v3/0,
					v1/b/15,
                        b/v2/0
				}{\path[draw, red, line width=0.8] (\i) edge[bend left=\t] (\j);}
                \foreach \i/\j/\t in {
					v4/a/15,
					a/v2/15,
					v2/v3/0,
					v3/b/15,
                        v1/v2/0
				}{\path[draw, green, line width=0.8] (\i) edge[bend left=\t] (\j);}
			\end{tikzpicture}
			\caption*{$(v)^*\times(i)$}
		\end{minipage}}	
		\subfigure{\begin{minipage}[t]{0.23\linewidth}
			\vspace{0pt}
			\centering
			\begin{tikzpicture}[scale=0.8]
				\filldraw[black](0,0) circle (3pt)node[label=left:$v_1$](v1){};
				\filldraw[black](2,0) circle (3pt)node[label=right:$v_2$](v2){};
				\filldraw[black](0,-2) circle (3pt)node[label=left:$v_3$](v3){};
				\filldraw[black](2,-2) circle (3pt)node[label=right:$v_4$](v4){};
				\filldraw[red](-1,-1) circle (3pt)node[label=above:$b$](b){};
				\filldraw[red](3,-1) circle (3pt)node[label=above:$a$](a){};
				
				\foreach \i/\j/\t in {
					v3/v4/0,
					v4/v2/0,
					b/v1/15,				
					v1/a/0	
				}{\path[draw, line width=0.8] (\i) edge[bend left=\t] (\j);}
				
				\foreach \i/\j/\t in {
					a/v4/15,
					v4/v1/0,
					v1/b/15,
					b/v2/0,
                        v1/v3/0
				}{\path[draw, red, line width=0.8] (\i) edge[bend left=\t] (\j);}
                \foreach \i/\j/\t in {
					v3/b/15,
					v2/v3/0,
					v1/v2/0,
					v4/a/15,
                        a/v2/15
				}{\path[draw, green, line width=0.8] (\i) edge[bend left=\t] (\j);}
			\end{tikzpicture}
			\caption*{$(v)^*\times(iii)$}
		\end{minipage}}
		\subfigure{\begin{minipage}[t]{0.23\linewidth}
			\vspace{0pt}
			\centering
			\begin{tikzpicture}[scale=0.8]
				\filldraw[black](0,0) circle (3pt)node[label=left:$v_1$](v1){};
				\filldraw[black](2,0) circle (3pt)node[label=right:$v_2$](v2){};
				\filldraw[black](0,-2) circle (3pt)node[label=left:$v_3$](v3){};
				\filldraw[black](2,-2) circle (3pt)node[label=right:$v_4$](v4){};
				\filldraw[red](-1,-1) circle (3pt)node[label=above:$b$](b){};
				\filldraw[red](3,-1) circle (3pt)node[label=above:$a$](a){};
				
				\foreach \i/\j/\t in {
					v2/v4/0,
					v3/v4/0,
					v3/a/0,
					b/v1/15
				}{\path[draw, line width=0.8] (\i) edge[bend left=\t] (\j);}
				
				\foreach \i/\j/\t in {
					a/v4/15,
					v4/v1/0,
                        v1/b/15,
                        b/v2/0,
                        v1/v3/0
				}{\path[draw, red, line width=0.8] (\i) edge[bend left=\t] (\j);}
				
				\foreach \i/\j/\t in {
					v4/a/15,
					a/v2/15,
					v2/v3/0,
					v3/b/15,
                        v1/v2/0
				}{\path[draw, green, line width=0.8] (\i) edge[bend left=\t] (\j);}

			\end{tikzpicture}
			\caption*{$(v)^*\times(v)$}
		\end{minipage}}
		
	\end{figure}
We are done now.

\end{proof}

\subsection{Semicomplete Split Digraph with Specific Structure}

If $D$ has structures illustrated in Theorem~\ref{thm:2as}\ref{ce2}, then observing that $u\in V_1,c\in V_2$, and $uc,cu\notin A$, we have $D$ is not a semicomplete split digraph.

If $D$ has structures illustrated in Theorem~\ref{thm:2as}\ref{ce1}, and $D$ is semicomplete split digraph, then:
\begin{enumerate}
    \item When $u\in V_1$: Since all arcs incident to $u$ has been illustrated and $D$ is a semicomplete split digraph, we have $V_2=\{x_1,x_2,x_3\}$. If there is another vertex $t\in V_1\setminus\{u\}$, then since $D$ is 2-arc-strong, $t$ has at least 2 in-arcs, which implies $t$ can only be $v$. 
    \begin{itemize}
        \item If $V_1=\{u\}$, then since $D$ is 2-arc-strong semicomplete digraph, it means either $D$ has strong arc decomposition or $D$ is $S_4$. 
        

        \item If $V_1=\{u,v\}$, then  $x_3x_1,x_2x_1\in A(D)$ since $N_D^+(x_1)=\{x_2,u\}, N_D^+(x_2)=\{v,u\}$. Since $x_3$ has at least 2 in-arcs, we have $vx_3\in A(D)$. This structure has been discussed in the Appendix of~\cite{ai2024}, the result of which is either $D$ has strong arc decomposition or $D$ is a counterexample listed in Appendix, and we have proved that $D$ always has good $(u,v)$-pair for all $u,v\in D$.

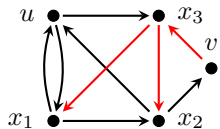
        \begin{figure}[H]
            \centering
            \begin{tikzpicture}[scale=0.7]
				\filldraw[black](0,0) circle (3pt)node[label=left:$u$](v1){};
				\filldraw[black](2,0) circle (3pt)node[label=right:$x_3$](v2){};
				\filldraw[black](0,-2) circle (3pt)node[label=left:$x_1$](v3){};
				\filldraw[black](2,-2) circle (3pt)node[label=right:$x_2$](v4){};
				
				\filldraw[black](3,-1) circle (3pt)node[label=above:$v$](a){};
				
				\foreach \i/\j/\t in {
					v1/v2/0,
					v3/v4/0,
					v4/v1/0,
					v1/v3/15,
                        v3/v1/15, 
					v4/a/0
				}{\path[draw, line width=0.8] (\i) edge[bend left=\t] (\j);}

                    \foreach \i/\j/\t in {
					v2/v3/0,
					v2/v4/0,
                        a/v2/0
				}{\path[draw, line width=0.8, red] (\i) edge[bend left=\t] (\j);}
		\end{tikzpicture}
        \end{figure}
    \end{itemize}

    \item When $u\in V_2$: Since $u$ is only adjacent to $x_1,x_2,x_3$, and the subdigraph induced by $V_2$ is semicomplete, then we have $|V_2|\leq 4$. Every vertex in $V_1$ is adjacent to $u$ since $D$ is semicomplete split digraph and so we have $V_1\cup V_2=\{u,x_1,x_2,x_3\}$.
    \begin{itemize}
        \item If $|V_2|=1$, then $x_1x_2\in D$ contradicts the facts that $V_1$ is an independent set.

        \item If $|V_2|=2$, then at least two of $x_1,x_2,x_3$ belongs to $V_1$. Since $D$ is 2-arc-strong, each vertex $x$ in $V_1$ has at least 2 in-arcs and 2 out-arcs, which means $x$ has 2-cycle with $u$. And so there would be two 2-cycles incident to $u$, which contradicts the structures as illustrated. 
        
        \item If $|V_2|=3,4$, then $D$ is 2-arc-strong semicomplete digraph, which means either $D$ has strong arc decomposition or $D$ is $S_4$. 

    \end{itemize}
\end{enumerate}
So, we are done.

\section{Remarks}
In this paper, we have considered good pairs in semicomplete split digraphs. Based on the results from~\cite{ai2024}, we only need to examine a finite number of graphs. In addition to the proofs above, we also verified all the possibilities with the assistance of a computer, and we have included the code in the Appendix.

\section{Appendix}
\begin{lstlisting}[language = R,breaklines = true]

install.packages("igraph")
library(igraph)
set.seed(124)

# main functions 
check_in_out_branching= function(edge_set,D,sample_size,details=F){
  vertex_set = unique(edge_set)
  l = length(vertex_set)
  check_matrix = matrix(0,nrow = l,ncol = l)
  colnames(check_matrix) = vertex_set
  rownames(check_matrix) = vertex_set
  tmp <- tempfile(pattern="image",tmpdir=".",fileext = ".jpg")
  jpeg(tmp, width=1000*l, height=1000*l)
  par(mfrow = c(l,l),cex = 3)
  for(i in 1:l){
    for(j in 1:l){
        for(k in 1:sample_size){
          t_ = subgraph.edges(D,sample_spanning_tree(D))
          t_check = dominator_tree(t_,vertex_set[i],mode = "in")
          if(length(t_check$leftout)==0){
            D_ = difference(D,t_)
            t_out = dominator_tree(D_,vertex_set[j],mode = "out")
            if(length(t_out$leftout)==0){
              check_matrix[i,j]=1
              if(details==T){
                print(plot(t_check$domtree,main =paste0(vertex_set[i],"-in-branching")))
                print(plot(t_out$domtree,main =paste0(vertex_set[j],"-out-branching")))
              }
              break
            }
          }
        } 
    }
  }
  dev.off()
  return(check_matrix)}    #details(default: False): whether to create a new .JPG file to show all the in-branchings and out-branchings

judge_in_out = function(check_matrix){
  a = nrow(check_matrix)
  all(check_matrix)!=0)
}


#input all the cases that should be checked
common_set = c("v1","v2","v2","v4","v3","v4","v1","v3","v4","v1","v2","v3")
a1_type1 = c("v2","a1","a1","v2","a1","v4","v4","a1")
a1_type2 = c("v2","a1","a1","v2","a1","v1","v4","a1")
a1_type3 = c("a1","v2","v1","a1","a1","v4","v4","a1")
a1_type4 = c("v2","a1","a1","v2","a1","v3","v4","a1")
a1_type5 = c("a1","v2","v3","a1","a1","v4","v4","a1")

a2_type1 = c("a2","v1","v1","a2","a2","v3","v3","a2")
a2_type2 = c("a2","v1","v4","a2","v3","a2","a2","v3")
a2_type3 = c("v1","a2","a2","v1","a2","v4","v3","a2")
a2_type4 = c("a2","v1","v2","a2","a2","v3","v3","a2")
a2_type5 = c("a2","v1","v1","a2","a2","v2","v3","a2")

edge_set1 = c(common_set,"v3","v1",a1_type1)

edge_set2 = c(common_set,"v3","v1",a1_type2)

edge_set3 = c(common_set,"v3","v1",a1_type3)

edge_set4 = c(common_set,"v3","v1",a1_type4)

edge_set5 = c(common_set,"v3","v1",a1_type5)

edge_set6 = c(common_set,a1_type1,a2_type1)

edge_set7 = c(common_set,a1_type1,a2_type2)

edge_set8 = c(common_set,a1_type1,a2_type3)

edge_set9 = c(common_set,a1_type1,a2_type4)

edge_set10 = c(common_set,a1_type2,a2_type2)

edge_set11 = c(common_set,a1_type2,a2_type3)

edge_set12 = c(common_set,a1_type2,a2_type4)

edge_set13 = c(common_set,a1_type3,a2_type3)

edge_set14 = c(common_set,a1_type3,a2_type4)

edge_set15 = c(common_set,a1_type3,a2_type5)

edge_set16 = c(common_set,a1_type4,a2_type4)



for(i in 1:16){
  edge_set = get(paste0("edge_set",as.character(i)))
  D = graph(edges = edge_set ,directed = T)
  check_matrix = check_in_out_branching(edge_set,D,2000)
  print(check_matrix)
  print(judge_in_out(check_matrix))
}


\end{lstlisting}

\end{document}